\documentclass[12pt,reqno]{amsart}

\usepackage{amssymb,hyperref}

\numberwithin{equation}{section}

\newtheorem{theorem}{Theorem}[section]
\newtheorem{prop}[theorem]{Proposition}
\newtheorem{lemma}[theorem]{Lemma}

\begin{document}

\title[Stable Higgs bundles on positive elliptic fibrations]{Stable Higgs bundles over 
positive principal elliptic fibrations}

\author{Indranil Biswas}

\address{School of Mathematics, Tata Institute of Fundamental Research, Homi 
Bhabha Road, Mumbai 400005, India}

\email{indranil@math.tifr.res.in}

\author{Mahan Mj}

\address{School of Mathematics, Tata Institute of Fundamental Research, Homi 
Bhabha Road, Mumbai 400005, India}

\email{mahan@math.tifr.res.in}

\author{Misha Verbitsky}

\address{Instituto Nacional de Matem\'atica Pura e
Aplicada (IMPA) Estrada Dona Castorina, 110,
Jardim Bot\^anico, CEP 22460-320,
Rio de Janeiro, RJ, Brasil, and, Laboratory of Algebraic Geometry, 
Faculty of Mathematics, National Research University HSE,
6 Usacheva street, Moscow, Russia}

\email{verbit@verbit.ru}



\date{\today}

\thanks{The first and second authors acknowledge the support of their respective J. C. Bose Fellowships. The 
third author is partially supported by the Russian Academic Excellence Project '5-100'.}

\begin{abstract}
Let $M$ be a compact complex manifold of dimension at least three and $\Pi\, :\, M\,\longrightarrow\, X$ a 
positive principal elliptic fibration, where $X$ is a compact K\"ahler orbifold. Fix a preferred Hermitian 
metric on $M$. In \cite{V}, the third author proved that every stable vector bundle on $M$ is of the form 
$L\otimes \Pi^*B_0$, where $B_0$ is a stable vector bundle on $X$, and $L$ is a holomorphic line bundle on 
$M$. Here we prove that every stable Higgs bundle on $M$ is of the form
$(L\otimes \Pi^*B_0,\, \Pi^*\Phi_X)$, where  $(B_0,\, \Phi_X)$ is a stable Higgs bundle on $X$ and
$L$ is a holomorphic line bundle on $M$.
\end{abstract}

\maketitle

\tableofcontents

\section{Introduction}

Stable holomorphic vector bundles on compact K\"ahler manifolds and, more generally, on compact Hermitian 
manifolds are extensively studied. On of the basic results in this topic is that an indecomposable 
holomorphic vector bundle $E$ on a compact complex manifold $X$ equipped with a Gauduchon metric $g_X$ 
admits a Hermitian structure $h_E$ that solves the Hermitian--Einstein equation if and only if $E$ is stable 
\cite{UY}, \cite{Do2}, \cite{LY}, \cite{LT1}.

A principal elliptic fibration is a holomorphic principal bundle with an elliptic curve as the structure 
group. Let $(X,\, \omega_X)$ be a compact K\"ahler orbifold. A principal elliptic fibration $\Pi\, :\, 
M\,\longrightarrow\, X$ is called positive if the differential form $\Pi^*\omega_X$ is exact. For a positive 
principal elliptic fibration $(M,\, \Pi)$, the complex manifold $M$ is never K\"ahler. There is a notion of
preferred Hermitian metric on the total space $M$ of a positive
principal elliptic fibration $(M,\, \Pi)$ which depends on the
K\"ahler form $\omega_X$ defined on the base $X$ of the fibration (the definition of
a preferred Hermitian metric is recalled in Section \ref{se2.1}).

Let $(M,\, \Pi)$ be a positive principal elliptic fibration equipped with a preferred Hermitian metric, such 
that $\dim_{\mathbb C} M\, \geq\, 3$. In \cite{V}, the third author investigated the stable vector bundles 
on $M$. The main result of \cite{V} says that every stable vector bundle $E$ on $X$ is of the form $L\otimes 
\Pi^* B_0$, where $L$ is a line bundle on $M$ and $B_0$ is a stable vector bundle on $(X,\, \omega_X)$.

A very important topic in the study of holomorphic vector bundles is the Higgs bundles. An indecomposable 
Higgs bundle $(E,\,\theta)$ is stable if and only if the vector bundle $E$ admits a Hermitian structure that 
solves the Yang--Mills equation associated to $(E,\,\theta)$ \cite{Si1}, \cite{LT2}.

Our aim here is to extend the result of \cite{V} mentioned above to the context of Higgs bundles.
Note that for a stable Higgs bundle $(E,\,\theta)$, the underlying vector bundle $E$ need not be stable.
The techniques of \cite{V} nevertheless can be adapted to our situation (as we do in Section 
\ref{Higgs_pullbacks_Section_}, see Theorem \ref{bundlepullsback}) to prove the analogous statement for the 
holomorphic vector bundle underlying the Higgs field. It turns out that more is true: the Higgs field also 
is a pullback. The following is the main theorem of our paper (see Theorem \ref{Higgspullsback}):

\begin{theorem}\label{introthm}
Let $\Pi\, :\, M\,\longrightarrow\, X$ be a positive principal elliptic fibration with
$n \,=\, \dim M \,\geq\, 3$. Equip $M$ with a preferred Hermitian metric. Let $(E,
\, \Phi)$ be a
stable Higgs bundle on $M$. Then there is stable Higgs bundle
$(B_0,\, \Phi_X)$ on $X$ and a holomorphic line bundle $L$ over $M$ such that the Higgs bundle
$(L\otimes \Pi^*B_0,\, \Pi^*\Phi_X)$ on $M$ is isomorphic to $(E, \, \Phi)$.
\end{theorem}

In \cite{BM}, the first two authors had proved an analog of Theorem \ref{introthm} for Sasakian manifolds 
where the fibers are {\it real} one dimensional instead of complex one-dimensional as in the situation here. 
The present paper thus gives a natural context where the ideas of \cite{V} and \cite{BM} come together.

\section{Preliminaries}

\subsection{Positive principal elliptic fibrations}\label{se2.1}

Let $T$ be a complex smooth elliptic curve with a point $0\, \in\, T$ which is
the identity element for the group structure on $T$. Let $(X,\, \omega_X)$ be a
compact K\"ahler orbifold with $\dim_{\mathbb C} X\, \geq\, 2$; the
K\"ahler form on $X$ is denoted by $\omega_X$. Take a pair $(M,\, \Pi)$, where
\begin{itemize}
\item $M$ is a compact complex manifold equipped with a holomorphic action of $T$
on the right of it, and

\item
\begin{equation}\label{Pi}
\Pi\, :\, M\, \longrightarrow\, X
\end{equation}
is a holomorphic surjective submersion
that commutes with the trivial action of $T$ on $X$; moreover, the
projection $\Pi$ and the action of $G$ on $M$ together make $M$ a holomorphic
principal $T$--bundle over the orbifold $X$.
\end{itemize}

Triples $(X,\, M,\, \Pi)$ of the above type are called \textit{principal $T$-fibrations},
and when $T$ is not fixed beforehand, quadruples $(T,\, X,\, M,\, \Pi)$
of the above type are called \textit{principal elliptic fibrations}; see
\cite[p.~251]{V}.

Note that the action of $T$ on $M$ is free only when the orbifold structure of
$X$ trivial (meaning $X$ is a K\"ahler manifold). If $X$ is a nontrivial 
orbifold, then the action of $T$ on $M$ can have finite isotropies on a closed analytic
subspace of $M$. 

We assume that the fibration $\Pi$ in \eqref{Pi} is \textit{positive}, meaning the 
$(1,1)$-form $\Pi^*\omega_X$ on $M$ is exact (see \cite[p.~253, 
Definition~2.1]{V}).

It is known that if the total space of an principal elliptic fibration over a
compact K\"ahler manifold is K\"ahler, then the principal elliptic fibration is
in fact isotrivial (pullback to some finite cover of base
is a trivial fibration). On the other hand, given a trivial fibration and a differential
form $\alpha$ on the base, if the pullback of $\alpha$ to the total space is exact,
then $\alpha$ is exact. Therefore, the positivity of $\Pi$ implies that $M$ is not
a K\"ahler manifold.

We note that an important class of positive principal elliptic fibrations is
quasi-regular Vaisman manifolds \cite[p.~253, Example~2.2]{V}, \cite{DO}, \cite{OV}.

Take a positive principal $T$-fibration $(X,\, M,\, \Pi)$. A Hermitian metric $g$ on
$M$ is called \textit{preferred} if
\begin{itemize}
\item $g$ is preserved by the action of $T$ on $M$, and

\item the map $\Pi$ from $(M,\, g)$ to $(X,\, \omega_X)$ is a Riemannian submersion,
meaning the restriction of the differential $d\Pi\, :\, TM\, \longrightarrow\, TX$ of $\Pi$
to the orthogonal complement $\text{kernel}(d\Pi)^\perp\, \subset\, TM$ is an isometry.
\end{itemize}
(See \cite[p.~254, Definition~2.6]{V}.)

A Hermitian metric on $g_Z$ on a complex manifold $Z$ of dimension $\delta$ is
called {\it Gauduchon} if the corresponding $(1,1)$-form $\omega_Z$ satisfies the
equation
$$
\partial\overline{\partial} \omega^{\delta-1}_Z\,=\, 0
$$
\cite{Ga}.

It is known that a preferred metric on a positive principal $T$-fibration is
Gauduchon \cite[p.~257, Proposition~4.1]{V}.

\subsection{Yang--Mills connections}

Let $Y$ be a compact complex manifold of dimension
$\delta$. A {\it Higgs field} on a holomorphic vector bundle $V$ over
$Y$ is a holomorphic section
$$
\Phi\, \in\, H^0(Y,\, \text{End}(V)\otimes \Omega^1_Y)
$$
such that the section $\Phi\wedge\Phi$ of $\text{End}(V)\otimes \Omega^2_Y$ vanishes
identically \cite{Si1}, \cite{Si2}. A {\it Higgs bundle} on $Y$ is a holomorphic vector bundle
on $Y$ equipped with a Higgs field.

Fix a Gauduchon metric $g_Y$ on $Y$; let $\omega_Y$ be the corresponding $(1,1)$-form.
The degree of a sheaf on $Y$ will be defined using $\omega_Y$. For a torsion-free coherent
analytic sheaf $F$ on $Y$, the real number $\text{degree}(F)/\text{rank}(F)$, which is
called the slope of $F$, is denoted by $\mu(F)$.

A Higgs bundle $(V,\, \Phi)$ is called \textit{stable} (respectively, \textit{semistable})
if for every coherent analytic subsheaf $F\, \subsetneq\, V$ such that $V/F$ is torsion-free,
and $\Phi(F)\, \subset\, F\otimes \Omega^1_Y$, the inequality
$$
\mu(F)\, <\, \mu(V)\ \ \text{(respectively,\ $\mu(F)\, \leq\, \mu(V)$)}
$$
holds. A Higgs bundle $(V,\, \Phi)$ is called \textit{polystable} if it is semistable and
isomorphic to a direct sum of stable Higgs bundles.

Let $V$ be a Hermitian holomorphic vector bundle over $Y$. The Chern connection on $V$
will be denoted as $\nabla$. Let
$$\Theta\,\in \,C^\infty(Y;\, {\rm End}(V)\otimes\Omega^{1,1}_Y)$$
be the curvature of $\nabla$. Define the
operator $$\Lambda\,:\; {\rm End}(V)\otimes\Omega^{1,1}_Y\,\longrightarrow\, {\rm End}(V)$$
to be the Hermitian adjoint of the multiplication operator $b \,\longmapsto\, b\otimes\omega$. 
Let $\Phi$ be a Higgs field on $V$.
The connection $\nabla$ is called  {\bf Yang-Mills} if $$\Lambda
(\Theta + [\Phi,\, \Phi^*])\,=\, c \cdot {\rm Id}_V$$
for some $c\, \in\,\mathbb C$ \cite{Si1}, \cite{Si2}.

It is known that given a Higgs bundle $(V,\, \Phi)$, the vector bundle $V$ admits a
Hermitian structure that satisfies the Yang-Mills equation for $(V,\, \Phi)$, if and only
if $(V,\, \Phi)$ is polystable \cite{LT2}, \cite{LT1}, \cite{Si1}.

\section{Vector bundles underlying stable Higgs bundles are pullbacks}
\label{Higgs_pullbacks_Section_}

As before, $(X,\, \omega_X)$ is a compact K\"ahler orbifold, and $T$ is an elliptic
curve with a base point.
Let $M$ be a compact complex manifold of dimension $n$, with $n\, \geq\, 3$.
Let
\begin{equation}\label{e1}
\Pi\, :\, M\, \longrightarrow\, X
\end{equation}
be a positive principal $T$-fibration. Let
\begin{equation}\label{tv}
T_v\, \subset\, TM
\end{equation}
be the line subbundle given by the action of $T$ on $M$. So $T_v$ is the
vertical tangent bundle for the projection $\Pi$.

Fix a preferred Hermitian metric $g$ on $M$. 
Let $\omega$ be the $(1,\, 1)$--form on $M$ associated to
$g$. This form $\omega$ is not closed, because $g$ is not K\"ahler \cite{Ho}.

The orthogonal complement $T^\perp _v$
will be denoted by $T_h$. We have
an orthogonal decomposition
$$
\omega\, =\, \omega_v+\omega_h\, ,
$$
where $\omega_v$ vanishes on $T_h$ and $\omega_h$ vanishes on $T_v$.

For a point $y\, \in\, M$, choose an orthonormal basis $\theta_0,\, \theta_1,\,
\cdots ,\,  \theta_{n-1}$ of 
$(T^{1,0}_yM)^*$, such that $\theta_0\,\in\,\omega_v(y)$ and $\theta_j\,\in\,
\omega_h(y)$ for $1\, \leq\, j\,\geq\ n-1$. Then as in \cite{V},
\begin{itemize}
\item $\omega \,=\, \sqrt{-1} \sum_{j\geq 0}\theta_j \wedge \overline{\theta_j}$, and
\item $\Pi^*\omega_X \,= \,\sqrt{-1} \sum_{j\geq 1}\theta_j \wedge \overline{\theta_j}$.
\end{itemize}

Let $E$ be a Hermitian holomorphic vector bundle on $M$. Let $\nabla$ be the Chern
connection on $E$; the curvature of $\nabla$ will be denoted by $\Theta$.
Note that $\Theta\,\in\, C^\infty(M,\, \Omega^{1,1}_M\otimes \text{End}(E))$
such that $\nabla(\Theta)\,=\, 0$ (Bianchi identity), and $\Theta^*\,=\, \Theta$;
we recall that $(A\otimes \alpha)^*\,=\, A^*\otimes\overline{\alpha}$ for any $A\, \in\,
\text{End}(E_y)$ and $\alpha\, \in\, \Omega^{p,q}(M)_y$.
Therefore, $\Theta(y)$ uniquely decomposes as
$$\Theta(y) \,=\, \sum_{i\neq j}(\theta_i \wedge \overline{\theta_j}
+ {\theta_j} \wedge \overline{\theta_i}) \otimes b_{ij}  +
\sum_{i}(\theta_i \wedge \overline{\theta_i} ) \otimes a_i\, , $$
where $a_i,\, b_{ij}\,\in\, \text{End}(E_y)$ with $a^*_i\,=\, -a_i$ and
$b^*_{ij}\,=\, -b_{ij}$.

Let $$\Xi\,:=\, Tr(\Theta\wedge \Theta)\, .$$ Then $\Xi$ is a closed
$(2,2)$ form which can be expressed as 
$$
  \sqrt{-1}^n\Xi \wedge \omega^{n-2} \,=\,
  Tr(-\sum_{j\geq 1} b_{0j}^2 +a_0 \sum_{j\geq 1} a_j)\,.
$$ 

We recall that a Higgs field on $E$ is a holomorphic section $\Phi$ of
$\text{End}(E)\otimes\Omega^{1,0}_M$ such that $\Phi\wedge\Phi\, =\, 0$. A
Hermitian structure on $E$ is said to satisfy the Yang-Mills equation for
$(E,\, \Phi)$, if the Chern connection $\nabla$ on $E$ satisfies the equation
\begin{equation}\label{hym}
\Lambda_\omega (\Theta + \Phi \wedge {\overline {\Phi}}) \,=\, 0\, ,
\end{equation}
where $\Theta$ as before denotes the curvature of $\nabla$.

Let $\Phi$ be a Higgs field on $E$ equipped with a Hermitian structure
that satisfies the Yang-Mills equation for $(E,\, \Phi)$ in \eqref{hym}. As before,
the corresponding Chern connection will be denoted by $\nabla$ and the
curvature of $\nabla$ by $\Theta$.
Consider the connection $\nabla^\Phi\,=\, \nabla+\Phi + {\Phi}^*$ on $E$.
Then the curvature $\Theta^\Phi$ of $\nabla^\phi$ is given by
(see Lemmas 3.2 and 3.4 of \cite{Si1} for instance)
\begin{equation}\label{TP}
\Theta_\Phi\,=\,\Theta + \Phi \wedge {\Phi}^* + \nabla^{(1,0)} \Phi + 
\nabla^{(0,1)} ({\Phi}^*)\, .
\end{equation}
We will denote $\nabla^{1,0} \Phi$ and $\nabla^{0,1}({\Phi}^*)$
by $\alpha$ and $\beta$ respectively. Let
$$\Xi^\Phi \,:=\, Tr(\Theta^\Phi\wedge \Theta^\Phi)\, .$$
We note that $\Xi^\Phi$ is cohomologous to the discriminant of the bundle $E$,
that is, the characteristic class represented by the trace of the
square of its curvature, and hence
the cohomology class of $\Xi_\Phi$ vanishes.
But  $$\Xi_\Phi \,=\, \Xi + Tr (\Phi \wedge {\Phi}^* \wedge \Phi
\wedge {\Phi}^*) + 2\cdot Tr (\alpha \wedge \beta)\, ,$$  and the
signs of all three terms on the right agree with each other.

But, $\sum_j a_j \,=\, \Lambda \Theta$; hence $\sum_j a_j + \Phi \wedge{\Phi}^*
\,=\, \Lambda \Theta_\Phi \,=\, 0$. Consequently,
\begin{center}
$i^n\Xi_\Phi \wedge \omega^{n-2} \,=\, Tr(-\sum b_i^2 -a_0^2 )$\ PLUS non-negative terms.
\end{center}

Since  $$\int_M i^n\Xi \wedge \omega^{n-2}\,=\,
\int_M i^n\Xi_\Phi \wedge \omega^{n-2}\,=\,0\, ,$$ and $Tr(-a_0^2)$ is a
positive definite form on $\mathfrak{u}(E)$ (the space of
element $A$ of $\text{End}(E)$ such that $A^*\,=\, -A$), it follows that each
$b_i$ is zero as is $a_0$.

We now have now the following analog of \cite[Proposition 4.2]{V}.

\begin{prop}\label{vert}
Let $\Pi\, :\, M\,\longrightarrow\, X$ be a positive principal elliptic fibration 
with $n \,=\, \dim M \,\geq\, 3$. Equip $M$ with a preferred Hermitian metric. Let $(E,
\, \Phi)$ be a Higgs bundle with a Hermitian structure that satisfies \eqref{hym}.
Higgs field $\Phi$ and let $\Theta_\Phi 
\,\in\, \Lambda^{1,1}(M) \otimes {\rm End}(E)$ be the associated endomorphism-valued
(1,1)-form in \eqref{TP}. Then $$\Theta_\Phi(v, -) \,=\,0$$
for any vertical tangent vector $v\,\in\, T_v M$.
\end{prop}

Proposition \ref{vert} furnishes, as in \cite[Theorem 6.1]{V}, the following. Since 
the proof is a replica of \cite[Theorem 6.1]{V}, once Proposition \ref{vert} is in 
place, we omit it.

\begin{theorem}\label{bundlepullsback}
Let $\Pi\, :\, M\,\longrightarrow\, X$ be a positive principal elliptic fibration with
$n \,=\, \dim M \,\geq\, 3$. Equip $M$ with a preferred Hermitian metric. Let $E$ be a
holomorphic vector bundle on $M$ admitting a Higgs field
such that the resulting Higgs bundle is stable. Then $E\,\cong\, L\otimes \Pi^\ast B_0$, where $L$ is a
holomorphic line bundle on $M$, and $B_0$ is a holomorphic vector bundle on $X$.
\end{theorem}

\section{Higgs field is a pullback}\label{higgspullback}

Consider the set-up of Theorem \ref{bundlepullsback}. Let $T$ denote
the elliptic curve that is acting on $M$. Take $E$ as in Theorem \ref{bundlepullsback}. Since $$E\,\cong\,
L\otimes \Pi^\ast B_0\, ,$$ we have
$$
\text{End}(E)\,\cong\, \Pi^*\text{End}(B_0)\, .
$$
Fix an isomorphism of $\text{End}(E)$ with $\Pi^*\text{End}(B_0)$.

The group $T$ has a tautological action on the pullback $\Pi^*\text{End}(B_0)$. 
Using the isomorphism of $\text{End}(E)$ with $\Pi^*\text{End}(B_0)$, this action of 
$T$ on $\Pi^*\text{End}(B_0)$ produces an action of $T$ on $\text{End}(E)$. On
the other hand, the action of $T$ on $M$ has a tautological lift to an action
of $T$ on the cotangent bundle $\Omega^1_M$. Combining the actions of $T$ on
$\text{End}(E)$ and $\Omega^1_M$, we get an action of $T$ on
$H^0(M,\, \text{End}(E)\otimes\Omega^1_M)$.

Since the group $T$ is compact, it does
not have any nontrivial holomorphic homomorphism to $\text{GL}(m, {\mathbb C})$ for any
$m\, \geq\,1$. Hence the action of $T$ on $H^0(M,\, \text{End}(E)\otimes\Omega^1_M)$
is trivial. In particular, every Higgs field on $E$ is fixed by the above action of $T$
on $H^0(M,\, \text{End}(E)\otimes\Omega^1_M)$.

Let
\begin{equation}\label{eta}
\eta\, \in\, H^0(M,\, TM)
\end{equation}
be the holomorphic vector field on $M$ given by a nonzero element
of $\text{Lie}(T)$ using the action of $T$ on $M$. So $\eta$ generates the line subbundle $T_v$ in \eqref{tv}.

Take a Higgs field $\Phi$ on $E$ such that the Higgs bundle $(E,\, \Phi)$ is stable.
Since $\Phi$ is a holomorphic section 
of $\text{End}(E)\otimes\Omega^1_M$, it follows that the contraction of $\Phi$ by the
vector field $\eta$ in \eqref{eta} is a holomorphic section
\begin{equation}\label{iP}
i_\eta \Phi\, \in\, H^0(M,\, \text{End}(E))\, .
\end{equation}

\begin{lemma}\label{le1}
The section $i_\eta\Phi$ in \eqref{iP} vanishes identically.
\end{lemma}

\begin{proof}
Since $\Phi\bigwedge\Phi\, =\, 0$, it follows immediately $\Phi$ commutes with
$i_\eta\Phi$. (We recall that the condition $\Phi\bigwedge\Phi\, =\, 0$ means that for
any point $y\, \in\, M$, and any $v,\, w\, \in\, T_y M$, the two elements
$i_v \Phi(y)$ and $i_w \Phi (y)$ of $\text{End}(E_y)$ commute.) Now, the stability condition
of the Higgs bundle $(E,\, \Phi)$ implies that there is a complex number $\lambda$
such that
$$
i_\eta\Phi\, =\, \lambda\cdot \text{Id}_E\, .
$$
To prove the lemma it suffices to show that $\lambda\, =\, 0$.

Consider the holomorphic $1$-form
$$
\text{trace}(\Phi)\, \in\, H^0(M,\, \Omega^1_M)\, .
$$
Evidently, we have
\begin{equation}\label{ti}
i_\eta (\text{trace}(\Phi))\,=\, \text{trace}(i_\eta\Phi)\,=\, n\cdot\lambda\, ,
\end{equation}
where $n\,=\, \dim_{\mathbb C} M$.

On the other hand, all holomorphic $1$-forms on $M$ are pulled back from $X$. Indeed,
this follows from the positivity of the principal elliptic fibration
$\Pi\, :\, M\,\longrightarrow\, X$. Consequently, for any holomorphic $1$-form $\beta$
on $M$, we have
\begin{equation}\label{bv}
i_\eta \beta\, =\, 0\, .
\end{equation}

Applying \eqref{bv} to $\text{trace}(\Phi)$, from \eqref{ti} we conclude that
$\lambda\, =\, 0$.
\end{proof}

Now consider $\Phi$ as a holomorphic section of $\text{End}(\Pi^*B_0)\otimes\Omega^1_M$
using the chosen isomorphism of $\text{End}(E)$ with $\Pi^*\text{End}(B_0)\,=\,
\Pi^*\text{End}(B_0)$. Since $\Phi$ is fixed by the action of $T$, from Lemma \ref{le1}
it follows immediately that the section
$$
\Phi\, \in\, H^0(M,\, \text{End}(\Pi^*B_0)\otimes\Omega^1_M)\,=\,
H^0(M,\, (\Pi^*\text{End}(B_0))\otimes\Omega^1_M)
$$
descends to $X$, meaning there is a section
$$
\Phi_X\, \in\, H^0(X,\, \text{End}(B_0)\otimes\Omega^1_X)
$$
such that
$$
\Pi^*\Phi_X\,=\, \Phi\, .
$$

Since $\Phi\wedge\Phi\,=\, 0$, and $\Pi$ is dominant, it follows immediately
that $\Phi_X\wedge\Phi_X\,=\, 0$. In other words, $(B_0,\, \Phi_X)$ is a Higgs
bundle on $X$. It is straight-forward to check that the Higgs bundle
$(B_0,\, \Phi_X)$ is stable with respect to $\omega_X$. Indeed, if a
subsheaf $F$ of $B_0$ violates the stability condition for $(B_0,\, \Phi_X)$,
then the subsheaf $L\otimes \Pi^* F$ of $L\otimes \Pi^*B_0\,=\, E$
violates the stability condition for $(E,\, \Phi)$.

Consequently, we have following generalization of Theorem \ref{bundlepullsback}:

\begin{theorem}\label{Higgspullsback}
Let $\Pi\, :\, M\,\longrightarrow\, X$ be a positive principal elliptic fibration with
$n \,=\, \dim M \,\geq\, 3$. Equip $M$ with a preferred Hermitian metric. Let $(E,
\, \Phi)$ be a
stable Higgs bundle on $M$. Then there is stable Higgs bundle
$(B_0,\, \Phi_X)$ on $X$ and a holomorphic line bundle $L$ over $M$ such that the Higgs bundle
$(L\otimes \Pi^*B_0,\, \Pi^*\Phi_X)$ is isomorphic to $(E, \, \Phi)$.
\end{theorem}

Note that Theorem \ref{Higgspullsback} justifies the definition of a Higgs bundle
on a quasi-regular Sasakian manifold adopted in \cite{BM}.

\end{document}